\documentclass[10pt]{amsart}
\usepackage{enumerate,amsmath,amssymb,latexsym,
amsfonts, amsthm, amscd, mathabx}

\theoremstyle{plain}

\newtheorem{theorem}{Theorem}

\numberwithin{equation}{section}

\newcommand{\R}{\mathbb{R}}

\newcommand{\G}{\Gamma}

\addtocounter{section}{-1}

\begin{document}

\title {The closed graph theorem is the open mapping theorem}

\date{}

\author [R.S Monahan and P.L. Robinson]{R.S. Monahan and P.L. Robinson}

\address{Department of Mathematics \\ University of Florida \\ Gainesville FL 32611  USA }

\email[]{monahanrs@ufl.edu, paulr@ufl.edu}

 \keywords{}

\begin{abstract}

We extend the closed graph theorem and the open mapping theorem to a context in which a natural duality interchanges their extensions. 

\end{abstract}

\maketitle

\section*{Introduction} 

\medbreak 

It is a familiar fact that the two fundamentally important theorems to which we refer in our title are closely related, both to each other and to the Bounded Inverse Theorem. This general understanding is fairly represented in Pryce [5] page 146: speaking of the Closed Graph Theorem, the author says `We derive it as a consequence of the Open Mapping Theorem; but ... the process can be reversed, so that these two apparently very different theorems are equivalent'. It is more strikingly presented in Pietsch [4] page 43: `We now discuss a fundamental result, which has many facets. ... Once we have proved one of the above conclusions, the others follow immediately'; here, the ellipsis covers statements of the closed graph theorem and the open mapping theorem, along with the bounded inverse theorem. 

\medbreak 

Our purpose in this note is to record that when suitably generalized, the closed graph theorem (CGT) and the open mapping theorem (OMT) become more than `equivalent' in the generally understood sense that either may be derived from the other: their suitable generalizations are simply two sides of the same coin, related by a natural duality. The key step in our generalization is the replacement of linear operators from a Banach space $X$ to a Banach space $Y$ by linear relations between the spaces; that is, by linear subspaces of the Banach space $X \times Y$. The natural duality that mediates between the relational CGT and the relational OMT is simply the switching of factors in $X \times Y$. In addition to deriving these relational theorems from their classical versions, we provide a stand-alone proof. 

\medbreak 

\section*{CGT $\equiv$ OMT}

\medbreak 

Let $X$ and $Y$ be Banach spaces. In their traditional forms, the closed graph theorem and the open mapping theorem address a linear operator $L$ from $X$ to $Y$. The classical CGT gives a sufficient condition for $L$ to be continuous: explicitly, $L: X \to Y$ is continuous if its graph $\{ (x, Lx) : x \in X \}$ is a closed subset of $X \times Y$. The classical OMT gives a sufficient condition for $L$ to be an open mapping: explicitly, $L: X \to Y$ is open if it is both continuous and surjective. The classical bounded inverse theorem asserts that if $L$ is a linear isomorphism then its continuity entails that of its inverse. It is customary to prove all three of these theorems together, first establishing one of them as a consequence of the Baire category theorem and then deriving the others from that one. It is in this relatively informal sense that these three theorems are frequently spoken of as being `equivalent': each of them may be derived from each of the others. We claim that passage from the context of linear operators to the broader context of linear relations yields extended versions of CGT and OMT that are actually equivalent in a very precise sense: the canonical flip $X \times Y \leftrightarrow Y \times X$ interchanges them.

\medbreak 

Let $\G$ be a linear relation from $X$ to $Y$: that is, let $\G$ be a linear subspace of the Banach space $X \times Y$.

\medbreak 

When $B \subseteq Y$ we define 
$$\G B = \{ x \in X : (\exists b \in B) \; (x, b) \in \G \} \subseteq X.$$
In particular, the {\it domain} of $\G$ is the subspace 
$${\rm Dom} \, \G = \G Y \subseteq X.$$
When $A \subseteq X$ we define 
$$A \G = \{ y \in Y : (\exists a \in A) \; (a, y) \in \G \} \subseteq Y.$$
In particular, the {\it range} of $\G$ is the subspace 
$${\rm Ran} \, \G = X \G \subseteq Y.$$
Equivalently, ${\rm Dom} \, \G$ is the image of $\G$ under first-factor projection $\pi_X : X \times Y \to X$; likewise, ${\rm Ran} \, \G$ is the image of $\G$ under second-factor projection $\pi_Y : X \times Y \to Y$. 

\medbreak 

As a special case, $\G$ may be the graph of a linear operator $L : X \to Y$. In this case, $A \G \subseteq Y$ is the direct image of $A \subseteq X$ under $L$ and $\G B \subseteq X$ is the inverse image of $B \subseteq Y$ under $L$. 

\medbreak 

We now state and prove the versions of the closed graph theorem and open mapping theorem that are appropriate to relations. In these theorems, $\G$ will be a closed subspace of $X \times Y$ and hence a Banach space in its own right. Our first appproach to each theorem will be to show how it follows from the corresponding classical version for operators. 

\medbreak 

The `closed graph theorem' for relations is as follows. 

\medbreak 

\begin{theorem} \label{CGT}
Let $\G \subseteq X \times Y$ be a closed linear relation with ${\rm Dom} \, \G = X$. If $B \subseteq Y$ is open then $\G B \subseteq X$ is open. 
\end{theorem} 

\begin{proof} Notice that, as $\G \subseteq X \times Y$ is closed, the subspace $Y_0 = \{ y \in Y : (0, y) \in \G \}$ is closed. Now, let $x \in X = {\rm Dom} \, \G$: there exists $y \in Y$ such that $(x, y) \in \G$; if also $(x, y') \in \G$ then $(0, y' - y) = (x, y') - (x, y) \in \G$ so that $y' - y \in Y_0$.  The resulting well-defined linear map 
$$\ell : X \to Y/Y_0 : x \mapsto y + Y_0$$
has closed graph, as $\G$ is closed; by the classical CGT, $\ell$ is therefore continuous. As is readily verified, if $B \subseteq Y$ then $\G B$ is exactly the inverse image of $B + Y_0 \subseteq Y/Y_0$ under $\ell$. Finally, if $B$ is open then $ B + Y_0$ is open (quotient maps are open) so that $\G B$ is open ($\ell$ being continuous). 

\end{proof} 

\medbreak 

In particular, this theorem applies to the graph of a linear operator $L: X \to Y$: in this case, it says that if $L$ has closed graph then $L$ is continuous; that is, Theorem \ref{CGT} reduces exactly to the classical closed graph theorem. We remark that in the proof of Theorem \ref{CGT}, $Y_0 = 0 \G$ and $y + Y_0 = x \G$ as sets; though this circumstance does not appear to simplify matters to any appreciable degree. 

\medbreak 

The `open mapping theorem' for relations is as follows. 

\medbreak 

\begin{theorem} \label{OMT}
Let $\G \subseteq X \times Y$ be a closed linear relation with ${\rm Ran} \, \G = Y$. If $A \subseteq X$ is open then $A \G \subseteq Y$ is open. 
\end{theorem} 

\begin{proof} 
The restriction $\pi_Y|_{\G}$ of the (continuous) linear projection $\pi_Y : X \times Y \to Y$ is surjective because ${\rm Ran} \, \G = Y$; by the classical OMT, $\pi_Y|_{\G}$ is therefore open. Now, if $A \subseteq X$ is open then $(A \times Y) \cap \G \subseteq \G$ is open, whence the openness of the map $\pi_Y|_{\G}$ implies the openness of the set 
$$\pi_Y|_{\G} [(A \times Y) \cap \G] = A \G.$$ 
\end{proof} 

\medbreak 

In particular, we may recover the classical open mapping theorem: let $L : X \to Y$ be a continuous linear map that is surjective and let $\G$ be its graph; then $\G$ is closed with ${\rm Ran} \, \G = Y$ and Theorem \ref{OMT} implies that $L$ maps open sets to open sets. In connexion with the proof of Theorem \ref{OMT}, see also [2] Section 11 Problem E, which was noted after the present paper was completed.  

\medbreak 

It is at once clear that Theorem \ref{CGT} and Theorem \ref{OMT} are but two manifestations of one theorem: the flip $X \times Y \to Y \times X : (x, y) \mapsto (y, x)$ transforms the one into the other; otherwise said, the one theorem applied to a closed linear relation is the other theorem applied to its transpose (or dual). Accordingly, it is natural to seek a proof of this `one' theorem that favours neither the classical CGT nor the classical OMT. 

\medbreak 

Such an even-handed proof is furnished by the Zabreiko lemma, a suitable version of which we now recall. Say that the seminorm $p$ on the Banach space $Z$ is {\it countably subadditive} precisely when 
$$p( \sum_n z_n) \leqslant \sum_n p(z_n)$$
for each convergent series $\sum_n z_n$ in $Z$; of course, this inequality need only be checked when the series on the right is convergent.  

\medbreak 

\noindent 
{\bf Lemma (Zabreiko).} {\it On a Banach space, each countably subadditive seminorm is continuous}. 

\bigbreak 

See [6] for a contemporaneous English translation of the original Russian paper of Zabreiko. See also Megginson [3] for a more recent account of the lemma and some of its uses. Of course, a (weak) version of the Baire category theorem lies behind the lemma. 

\medbreak 

Our even-handed proof of Theorem \ref{CGT} and Theorem \ref{OMT} is as follows, taking the `one' theorem in its second manifestation. Let the linear relation $\G \subseteq X \times Y$ be closed and let ${\rm Ran} \, \G = Y$. Define $p : Y \to \R$ by the rule that if $y \in Y$ then 
$$p(y) = \inf \{ ||x|| : (x, y) \in \G \}.$$ 
Let the series $\sum_{n > 0} y_n$ be convergent in $Y$. When $\varepsilon > 0$ is given, choose $x_n \in \G y_n$ such that $||x_n || < p(y_n) + \varepsilon / 2^n$;  it follows that $\sum_n ||x_n || < \sum_n p(y_n) + \varepsilon$, whence the series $\sum_{n > 0} x_n$ is absolutely convergent and therefore convergent in the Banach space $X$. Denote the partial sums of the series $\sum_n x_n$ and $\sum_n y_n$ by  $x_{(n)} = x_1 + \dots + x_n$ and $y_{(n)} = y_1 + \dots + y_n$; denote the sums of the series by $x$ and $y$. As $\G$ is closed, from $(x_{(n)}, y_{(n)}) \in \G$ it follows that $(x, y) \in \G$; consequently, 
$$p(y) \leqslant ||x|| \leqslant \sum_n ||x_n|| < \sum_n p(y_n) + \varepsilon.$$  
The given $\varepsilon > 0$ being arbitrary, we deduce that $p$ is a countably subadditive seminorm. According to the Zabreiko lemma, $p$ is continuous. Let $X_1 \subseteq X$ be the open unit ball: then 
$$X_1 \G = \{ b \in Y : (\exists a \in X_1) (a, b) \in \G \} = \{ b \in Y : p(b) < 1 \} \subseteq Y$$
is open. Finally, let $U \subseteq X$ be an arbitrary open set and let $v \in U \G$; say $u \in U$ and $(u, v) \in \G$. If $r > 0$ is so chosen that $u + r X_1 = B_r(u) \subseteq U$ then $u \G + r X_1 \G \subseteq Y$ is open and (because $\G \subseteq X \times Y$  is linear) 
$$v \in u \G + r X_1 \G \subseteq (u + r X_1) \G = B_r(u) \G \subseteq U \G.$$
This proves that if $U$ is open then $U \G$ is open, concluding our even-handed proof of Theorem \ref{OMT} or equivalently of Theorem \ref{CGT}. 

\medbreak 

We may take the even-handed approach further (perhaps too far). Let $\G \subseteq X \times Z \times Y$ be a linear {\it ternary} relation between Banach spaces. When $(x, y) \in X \times Y$ we shall write 
$$x \G y = \{ z \in Z : (x, z, y) \in \G \};$$
when $W \subseteq Z$ we shall write 
$$\G W \G = \{ (x, y) \in X \times Y : W \cap \, x \G y \neq \emptyset \}.$$ 

\medbreak 

With this understanding, the `closed graph theorem' and the `open mapping theorem' are special cases of the following theorem. 

\medbreak 

\noindent
{\bf Theorem.} {\it Let the projection $X \times Z \times Y \to X \times Y$ be surjective when restricted to the closed linear subspace $\G \subseteq X \times Z \times Y$. If $W \subseteq Z$ is open then $\G W \G \subseteq X \times Y$ is open.}

\begin{proof} 
Transplant the proof displayed ahead of the theorem, with appropriate modifications. The rule 
$$p(x, y) = \inf \{||z|| : z \in x \G y \}$$
defines on $X \times Y$ a seminorm that is countably subadditive and hence (Zabreiko) continuous. If $Z_1$ is the open unit ball in $Z$ then 
$$\G Z_1 \G = \{ (x, y) : p(x, y) < 1 \}$$ 
is therefore open in $X \times Y$. Finally, let $W \subseteq Z$ be open and let $(x, y) \in \G W \G$; choose $w \in W \cap \, x \G y$ and choose $r > 0$ so that $w + r Z_1 \subseteq W$; then  
$$(x, y) \in \G w \G + r \G Z_1 \G \subseteq  \G (w + r Z_1) \G \subseteq \G W \G$$  
whence we conclude that $\G W \G \subseteq X \times Y$ is open. 
\end{proof} 

\medbreak 

When $Y = 0$ we recover Theorem \ref{CGT}: dropping the zero factors, the projection $X \times Z \to X$ is surjective on $\G$ and $\G W \G$ becomes $\G W$. 

\medbreak 

When $X = 0$ we recover Theorem \ref{OMT}: dropping the zero factors, the projection $Z \times Y \to Y$ is surjective on $\G$ and $\G W \G$ becomes $W \G$. 

\medbreak 

\section*{Remarks} 

\medbreak 

Our even-handed proof of Theorem \ref{OMT} is in essence the proof of the classical OMT for linear operators presented as Theorem 1.6.5 in Megginson [3] but modified for transport to the context of linear relations. 

\medbreak 

An alternative even-handed proof derives from notions related to convex series as presented in Section 22 of Jameson [1]; especially, Theorem 22.7 of [1] is clearly ripe for application to the present paper. 

\medbreak 

Lastly, we remark that wider territory can be examined from the point of view expressed in the present paper: for example, the closed graph theorem and the open mapping theorem are valid for complete metrizable topological vector spaces.  

\medbreak 

\bigbreak 

\begin{center} 
{\small R}{\footnotesize EFERENCES}
\end{center} 
\medbreak 

[1] G.J.O. Jameson, {\it Topology and Normed Spaces}, Chapman and Hall (1974). 
\medbreak 

[2] J.L. Kelley and I. Namioka, {\it Linear Topological Spaces}, Van Nostrand (1963). 

\medbreak 

[3] R.E. Megginson, {\it An Introduction to Banach Space Theory}, Graduate Texts in Mathematics {\bf 183}, Springer (1998). 

\medbreak 

[4] A. Pietsch, {\it History of Banach Spaces and Linear Operators}, Birkh\"auser (2007). 

\medbreak 

[5] J.D. Pryce, {\it Basic Methods of Linear Functional Analysis}, Hutchinson (1973). 

\medbreak 

[6] P.P. Zabreiko, {\it A theorem for semiadditive functionals}, Functional Analysis and its Applications {\bf 3} (1969) 70-72. 

\medbreak

\end{document}